\documentclass[reqno,a4paper,12pt,russian,english]{amsart}
\usepackage{amsmath,amsthm,amsfonts,amssymb}
\usepackage{verbatim, graphicx, ifthen, enumitem}
\usepackage[T1]{fontenc}
\usepackage [applemac,koi8-r,latin9] {inputenc}
\usepackage{upquote,bbm}%better than dsfont and bbold
\usepackage{babel}
\usepackage[bookmarksopen=false,pdfborder={0 0 0},pdfborderstyle={},colorlinks=true,linkcolor=black,citecolor=black,urlcolor=blue]{hyperref}

\allowdisplaybreaks

\let\oldmarginpar\marginpar
\renewcommand\marginpar[1]{\-\oldmarginpar[\raggedleft\footnotesize #1]%
{\raggedright\footnotesize #1}}
\theoremstyle{plain}

\newtheorem*{thm*}{Theorem}

\newenvironment{lyxlist}[1]
{\begin{list}{}
{\settowidth{\labelwidth}{#1}
 \setlength{\leftmargin}{\labelwidth}
 \addtolength{\leftmargin}{\labelsep}
 }}
{\end{list}}
\DeclareRobustCommand{\cyrtext}{%
  \fontencoding{T2A}\selectfont\def\encodingdefault{T2A}}
\DeclareRobustCommand{\textcyr}[1]{\leavevmode{\cyrtext #1}}
\def\clap#1{\hbox to 0pt{\hss#1\hss}}

\usepackage{calrsfs}
\usepackage{graphicx}
\usepackage{color}
\usepackage{perpage}
\gdef\SetFigFontNFSS#1#2#3#4#5{} %Silence pointless warnings due to xfig

\theoremstyle{remark}
\newtheorem*{qst*}{Question}

\newtheorem{claim}{Claim}

\newtheorem{lemma}{Lemma}

\newtheorem{corollary}{Corollary}
\theoremstyle{definition}

\theoremstyle{remark}
\newtheorem{remark}{Remark}
\newtheorem*{remark*}{Remark}

\newcommand{\Z}{\mathbb{Z}}

\newcommand{\N}{\mathbb{N}}
\newcommand{\R}{\mathbb{R}}

\newcommand{\CC}{\mathbb{C}}

\def\L{\Lambda}
\def\l{\lambda}

\def\N{\mathbb{N}}
\def\Z{\mathbb{Z}}

\def\R{\mathbb{R}}

\def\1{\mathbbm{1}}
\def\eps{\varepsilon}

\begin{document}

%\title[Riesz bases of for unions of intervals]{Riesz bases of exponentials for finite unions of intervals}
\title{Combining Riesz bases}
\author{Gady Kozma and Shahaf Nitzan}

%\begin{abstract}
%\end{abstract}
\maketitle
\section{Introduction}
Let $S\subset\mathbb{R}^{d}$ be some set, and $f\in L^{2}(S)$. How
can one represent $f$ as a combination of exponentials? Ideally,
one would like to find a sequence $\Lambda\subset\mathbb{R}^{d}$ such that
the functions $\{e^{i\langle\lambda,t\rangle}\}_{\lambda\in\Lambda}$
form an orthonormal basis of $L^{2}(S)$. Such an object would give
each $f$ a unique representation as $f(t)=\sum c_{\lambda}e^{i\langle\lambda,t\rangle}$,
and the coefficients would be easy to calculate. Unfortunately, orthonormal
bases of exponentials are not easy to come by. It goes back to Fuglede
%\marginpar{I added a sentence at the end of paragraph 1}
\cite{F74} that even the ball in two or more dimensions does not
enjoy an orthogonal basis of exponentials. The reason is that orthogonality
is too strong a requirement, it requires that each couple of exponentials
have their difference in the zero set of $\widehat{\1_{S}}$
(a Bessel function, in the case of the ball). See also
\cite{F01,IKT03,T04,KM06,S00}. A similar statement holds for some
simply constructed subsets of $\R$, even the union of as few as two
disjoint intervals may not admit an orthogonal basis of exponentials,
see e.g.\ \cite{IK}.

If one cannot find an orthonormal basis, a Riesz basis is the second
best (see \S\ref{sec:prelim} for precise definitions). A Riesz basis
%\marginpar{changed "would also give" to "also gives"}
also gives each function a unique representation $f(t)=\sum c_{\lambda}e^{i\langle\lambda,t\rangle}$
in a stable manner. Our understanding of the existence of Riesz bases
of exponentials is still lacking. On the one hand, there are relatively
few examples in which it is known how to construct a Riesz basis of
exponentials.
%\marginpar{changed "can be constructed" to "it is known how to construct"}
On the other hand, we know of no example of a set $S$ of positive
measure for which a Riesz basis of exponentials can be shown not to
exist. In particular the question is not settled for the ball in two
or more dimensions.

In this paper we prove the following.
\begin{thm*}
Let $S\subset\mathbb{R}$ be a finite union of intervals. Then there
%\marginpar{I added info to the statement of this theorem}
exists a set $\Lambda\subset\mathbb{R}$ such that the functions $\{e^{i\lambda t}\}_{\lambda\in\Lambda}$
form a Riesz basis in $L^{2}(S)$. Moreover, if $S\subseteq [0,2\pi]$ then $\L$ may be chosen to satisfy $\L\subseteq \Z$.
\end{thm*}
Interest in Riesz bases of exponentials for finite unions of intervals
%\marginpar{changed "union" to "unions"}
has its roots in practical applications to sampling of band-limited
signals, and the first partial results came from there. Thus Kohlenberg
\cite{K53} solved the case of two intervals of equal length. Bezuglaya
and Katsnelson \cite{BK93} solved the case that the intervals have
integer end points. Seip \cite{S95} did the general case of two intervals
and a few subcases of three or more intervals. Lyubarskii and Seip
\cite{LS97} give a well-written survey on the problem including a
solution of the case that the intervals have equal lengths (but arbitrary
positions), as well as a proof that a Riesz basis of exponential with
\emph{complex }$\lambda$ always exists. An interesting approach based
on \emph{quasi-crystals} \cite{MM10} allowed to construct
Riesz bases for additional families of unions of intervals, under an
arithmetic condition on the lengths
\cite{KL11,L12}. A control-theory approach was investigated in
\cite{AM99,ABM07} and a reduction to inversions of convolution
integral operators in \cite{K96,LS96}.

Due to the popularity the problem used to enjoy, the following conversation
must have repeated in many places and times:
%\marginpar{added : to the conversation, changed the advisors first answer, added "stability thm" to the students reply}
\begin{lyxlist}{000}
\item [{\textsc{Student:}}] Why is it difficult to construct a Riesz basis
for two intervals? Say their lengths are $\alpha$ and $\beta$, can't
you just take $(1/\alpha)\mathbb{Z}\cup(1/\beta)\mathbb{Z}$?
\item [{\textsc{Advisor:}}] Think of the case $\alpha=\beta$!
\item [{\textsc{Student:}}] OK, but you can move one of the copies a
  little, and then the union forms a Riesz basis, right?
\item [{\textsc{Advisor:}}] How about the case that $\alpha$ and
  $\beta$ are rationally independent? Then no matter how you move two
  copies, the union is not even separated.
\item [{\textsc{Student:}}] Right, just moving will not work in this case. But we can still use the Paley-Wiener stability theorem to perturb each
one a little (so it's still a Riesz basis) in such a way that when
you take the union it will be separated.
\item [{\textsc{Advisor:}}] But how do you show that this is a Riesz basis
for the union?
\item [{\textsc{Student:}}] Eh...
\end{lyxlist}
(see \S\ref{sec:prelim} for Paley-Wiener's stability theorem and the role of
separatedness). The advisor probably knew that the union of two Riesz basis for two
intervals is not necessarily a Riesz basis for the union, even if
they are separated. A simple example comes from taking
%\marginpar{changed "had to made"} by taking
$\{\dotsc,-1\frac{3}{4},-\frac{3}{4},0,\frac{3}{4},1\frac{3}{4},2\frac{3}{4},\dotsc\}$
which is \emph{not} a Riesz basis for $[0,2\pi]$ by Kadec \cite{K64},
and noting that taking the even entries in the series gives a Riesz
basis for $[0,\pi]$ while the odd entries form a Riesz basis for
$[\pi,2\pi]$, both, again, by the Kadec $\frac{1}{4}$ rule. Nevertheless,
the crux of our proof, see \S\ref{sec:union} below, is that under
certain conditions, one can actually take the union and get a Riesz
basis.

Before embarking with the details, we wish to draw the reader's
attention to an interesting connection between the problem of finding
\emph{some} Riesz basis of exponentials on a complicated set, and
characterizing Riesz bases of exponentials on a \emph{single} interval.
For example, Seip \cite{S95} translates the problem of finding a Riesz
basis for two intervals to the problem of finding a Riesz basis of
$[0,1]$ which is a subset of $\alpha\mathbb{Z}$ for some $\alpha<1$;
and \cite{KL11,L12} reduce
%\marginpar{changed "reduces" to "reduce", removed the last sentence in this paragraph}
the problem to asking whether certain
arithmetic sets are Riesz bases for a single interval. The case of
a single interval is amenable to techniques from complex analysis,
leading to deep results by Kadec \cite{K64}, Katsnelson \cite{K71}
Avdonin \cite{A79} and Pavlov \cite{P79,HNP80}. Our approach also
relies on a reduction to a single interval. However, in the interest
of self-containment, we will give a proof that does not use these
results, only the Paley-Wiener stability theorem (see \S\ref{sec:prelim}).

\vfill

\section{preliminaries}\label{sec:prelim}

\subsection{Systems of vectors in Hilbert spaces}

Let $H$ be a separable Hilbert space. A system of vectors $\{f_n\}\subseteq H$ is called a \textit{Riesz basis} if it is the image, under a bounded invertible operator, of an orthonormal basis. This means that $\{f_n\}$ is a Riesz basis if and only if it is complete in $H$ and satisfies the following inequality for all sequences $\{a_n\}\in l^2$,
\begin{equation}\label{riesz sequence}
c\sum|a_n|^2\leq \|\sum a_nf_n\|^2\leq C\sum |a_n|^2,
\end{equation}
where $c$ and $C$ are some positive constants. A system
$\{f_n\}\subseteq H$ which satisfies condition (\ref{riesz sequence}),
but is not necessarily complete, is called a \textit{Riesz sequence}.

A simple duality argument shows that $\{f_n\}$ is a Riesz basis if and only if it is minimal (i.e.\ no vector from the system lies in the closed span of the rest) and satisfies the following inequality for every $f\in H$,
\begin{equation}\label{frame}
c\|f\|^2\leq\sum |\langle f, f_n\rangle|^2\leq C\|f\|^2
\end{equation}
where $c$ and $C$ are some positive constants (in fact, the same constants as in (\ref{riesz sequence})). A system $\{f_n\}\subseteq H$ which satisfies condition (\ref{frame}), but is not necessarily minimal, is called a \textit{frame}.

In particular, this discussion implies the following:
\begin{lemma}\label{rb is rs and frame}
A system of vectors in a Hilbert space is a Riesz basis if and only if it is both a Riesz sequence and a frame.
\end{lemma}
%\marginpar{I changed the last paragraph here, they way it was written was not correct. Also added a proof, mainly because at this point it amuses me that we are proving every word we write...you are welcome to remove the proof}
If $\{f_n\}$ is a Riesz basis then there exists a
{\em dual system} $g_n$ with $\langle f_n,g_m\rangle = \1_{\{n=m\}}$ and
$\|g_n\|\le C$. Indeed, if $\{\phi_n\}$ is an orthonormal basis in $H$ and $T:H\mapsto H$ is a bounded invertible operator which satisfies $T\phi_n=f_n$ then the system $g_n:=(T^*)^{-1}\phi_n$ has the required properties.
\subsection{Density and systems of exponentials}

%\marginpar{changed the last line in the first paragraph }
We give a short and partial review of results connecting the
properties of a sequence of exponentials and the density of its
corresponding frequencies, due to Landau. These results will
\textbf{not} be used in any of our proofs. But they are useful to keep
in mind while reading \S 3 below.

Let $\L=\{\l_n\}$ be a \textit{separated sequence} of real numbers, i.e.
\[
|\l_n-\l_m|>\delta\qquad \forall n\neq m,
\]
for some positive constant $\delta$. Let us define the upper and lower
densities %The sequence $\L$ is called \textit{uniformly distributed} if%\marginpar{changed $\{$ to $|$ in the formula - for the cardinality of the sets}
\[
{D}^-(\L)=\lim_{r\rightarrow\infty}\frac{\min_{|I|=r}|\L\cap
  I|}{r}\qquad
{D}^+(\L)=\lim_{r\rightarrow\infty}\frac{\max_{|I|=r}|\L\cap I|}{r}.
\]
The sequence $\L$ is called \emph{uniformly distributed} if
${D}^-(\L)={D}^+(\L)$. In this case the common value is called the uniform density of $\L$ and denoted by $D(\L)$.

Given a set $S\subseteq \R$ of positive measure, consider the system of exponentials
\[
E(\L)=\{e_{\l}\}_{\l\in\L},\qquad e_\l(t)=e^{2\pi i\l t},
\]
in the space $L^2(S)$. The results of Landau give necessary conditions, in terms of the density of $\L$, for the system $E(\L)$ to be a Riesz sequence or a frame in the space \cite{L67a} (see also \cite{NO}).
\begin{thm*}[Landau]
Let $\L$ and $S$ be as above and assume that $\L$ is separated and $S$ is bounded.%\marginpar{removed the $2\pi$ from landau's thm - because we added it to the def of $e_l(t)$.}
\begin{itemize}
\item [i.] If $E(\L)$ is a Riesz sequence in $L^2(S)$ then ${D}^+(\L)\leq |S|$.
\item [ii.] If $E(\L)$ is a frame in $L^2(S)$ then ${D}^-(\L)\geq |S|$.
\end{itemize}
\end{thm*}%\marginpar{I changed the sentence just after landau's thm-added separation, also changed the beginning of remark 1}
If $E(\L)$ is a Riesz basis of $L^2(S)$ then $\L$ must
be separated (see below). We conclude from Landau's theorem that $\L$
is uniformly distributed with $D(\L)= |S|$.

\begin{remark}\label{bessel condition}%\marginpar{added fourier transform to the proof in this remark}
 If $E(\L)$ is a Riesz basis for a bounded set, then $\L$ must be
 separated. Indeed, if $\l,\mu\in\L$ and $|\l-\mu|<\epsilon$ then
 there exists a corresponding element of the dual system $g$ whose Fourier transform $G$ satisfies
 $G(2\pi\l)=1$ and $G(2\pi\mu)=0$. This implies that
 $\|G'\|_\infty>1/2\pi\epsilon$. Combining this with the fact that
 $g$ is supported on a bounded set leads to a contradiction to the
 fact that the dual system has bounded norms.

 On the other hand, for a bounded set $S$ the separation condition on $\L$ implies that
 the system $E(\L)$ satisfies the right inequality in (\ref{frame})
 for the space $L^2(S)$ (equivalently, the right inequality in
 (\ref{riesz sequence})). See, for example, \cite{Y01}.
\end{remark}

\subsection{Stability of systems of exponentials}

%\marginpar{I changed the first paragraph here}
In the first step of our construction we will use a theorem of Paley and Wiener regarding the stability of Riesz bases of exponentials over single intervals \cite{PW34}.
For the completeness of the paper we add a short proof of the theorem, which is applicable also for general spectra.
\begin{thm*}[Paley \& Wiener]%\marginpar{we can remove the separation requirements from both sequences in PW thm - they follow from the conditions - but it is not necessary-as you wish}
Let $S\subseteq \R$ be a bounded set of positive measure and $\L$ be a
sequence of real numbers such that $E(\L)$ is a Riesz basis in
$L^2(S)$. Then there exists a constant $\eta=\eta(S,\L)$ such that if a sequence $\Gamma=\{\gamma_n\}$ satisfies
\[
|\l_n-\gamma_n|<\eta\qquad\forall n
\]
then $E(\Gamma)$ is also a Riesz basis in $L^2(S)$.
\end{thm*}
\begin{proof}
It is well known (and easy to check) that if $\{f_n\}$ is a Riesz basis in a Hilbert space $H$ then there exists some constant $c$ such that every sequence $\{g_n\}\subseteq H$ which satisfies
\begin{equation}\label{stability condition}
\sum |\langle f,f_n-g_n\rangle |^2\leq c\|f\|^2\qquad \forall f\in H
\end{equation}
is also a Riesz basis %\marginpar{changed sequence to basis}
in $H$ (see \cite{Y01}).

To check this condition in our setting fix $f\in L^2(S)$ and denote its Fourier transform by $F$. Since $S$ is bounded, $F$ has a derivative on the real axis which is an image, under the Fourier transform, of some function in $L^2(S)$, say $h$. We have, $\|h\|\leq C\|f\|$, for some constant $C$  depending only on the diameter of $S$. Hence we get,%\marginpar{added 2$\pi$ in various places}
\begin{gather*}
\sum |\langle f, e_{\l_n}-e_{\gamma_n}\rangle|^2=\sum |F(2\pi\l_n)-F(2\pi\gamma_n)|^2\\
\le (2\pi\eta)^2 \sum|F'(2\pi\xi_n)|^2=(2\pi\eta)^2 \sum|\langle h, e_{\xi_n}\rangle|^2
\end{gather*}
where $\xi_n$ is some number between $\lambda_n$ and $\gamma_n$. Now,
by the first part of Remark \ref{bessel condition}, $\L$ must be
separated, and hence, if $\eta$ is sufficiently small, so will be the
sequence $\{\xi_n\}$. Using this with the second part of Remark \ref{bessel
  condition} gives that the last expression is smaller than
$C\eta^2\|h\|^2$. But $\|h\|\le C\|f\|$, for some constant $C$ depending only on the diameter of $S$. Hence, for small enough $\eta$ the condition in (\ref{stability condition}) is fulfilled.
\end{proof}

\begin{remark}
Both the theorem and its proof hold also for frames and Riesz sequences of exponentials.
\end{remark}
\begin{remark}
In the case that $S=[0,1]$ and $\L=\Z$ the theorem holds for every
constant $\eta<\frac14$, and this is sharp (this is the Kadec $\frac 14$
theorem already mentioned, see \cite{L40,K64}).
\end{remark}

\subsection{Conventions} We use $\{x\}$ to denote the fractional value
of $x$ i.e.\ the unique element in $(x+\Z)\cap [0,1)$. We use $c$ and $C$ to denote constants which
depend %\marginpar{changed depends to depend}
on the system of exponentials under considerations but do not
depend on other parameters. Their value may change from place to place
and even inside the same formula. We will number them occasionally for
clarity; numbered constants do not change their value. $C$ will be
used for constants sufficiently large and $c$ for constants
sufficiently small.

\section{the basic lemma}\label{sec:union}

%\marginpar{I added a "beginning" to the section}
In this section we prove a general lemma which is used in several places throughout our construction. This lemma describes how one can get a Riesz basis of exponentials over a set $S$, given that there exist Riesz bases over some sets related to $S$.

Fix a positive integer $N$. Given a set $S\subset [0,1]$, define
%\begin{align*}
%&A_n=\{t\in S: t+\frac{2\pi j}{N}\in S \textrm{ for exactly } n \textrm{ integer } j{'s}\}\qquad n=1,...,N\\
%&A_n^{j}=A_n\cap [\frac{2\pi j}{N}, \frac{2\pi (j+1)}{N}]\qquad n,j=1,...,N \\
%&B_n=\{t\in [0, 1]: t+\frac{j}{N}\in S \textrm{ for exactly } n \textrm{ integer } j{'s}\}\qquad n=0,1,...,N\\
%&E_n=\cup_{k=n}^{N}B_n\qquad n=0,1,...,N
%\end{align*}
\begin{align}
&A_n=\Big\{t\in\Big[0,\frac1N\Big]:t+\frac jN\in S\textrm{ for exactly $n$
  values of $j\in\{0,\dotsc,N-1\}$}\Big\}\nonumber\\
&A_{\ge n}=\bigcup_{k=n}^N A_k\label{eq:defAgen}
\end{align}
\begin{lemma}\label{lem:main}
If there exist $\L_1,\dotsc,\L_N\subseteq N\Z$ such that
the system $E({\L_n})$ is a Riesz basis in $L^2(A_{\ge n})$, then the system $E(\L)$, where
\[
\L=\bigcup_{j=1}^{N}(\L_j+j),
\]
is a Riesz basis in $L^2(S)$.%\marginpar{I made a slight change to the formulation of the lemma}
\end{lemma}

\begin{proof}%\marginpar{changed lemma to Lemma}
We will use lemma \ref{rb is rs and frame} and show that $E(\L)$ is
both a frame and a Riesz sequence. Throughout the proof we will use the notation
\[
e(t)=e_1(t)=e^{2\pi i t}.
\]
Clearly, it is equivalent to prove the lemma under the assumptions
that
\[
\L_n\subset N\Z+n\qquad \L=\bigcup_{j=1}^N\L_j
\]
(but still requiring that $\L_n$ is a Riesz basis for $A_{\ge n}$ ---
recall that the property of being a Riesz basis is invariant to translations) which will make the notations a little shorter.

\medskip
\noindent\textbf{Frame.} To show that $E(\L)$ is a frame in
$L^2(S)$ we need to show that for any $f\in L^2(S)$%\marginpar{changed c to b to avoid confusion later}
\[
\sum_{\l\in\L}|\langle f, e_\l\rangle|^2>c_1||f||^2
\]
(the right inequality in the definition of a frame, (\ref{frame}), is
satisfied because $S\subset[0,1]$ and $\L\subset\Z$). For $n\in\{1,\dotsc,N\}$, denote by $f_n$ the restriction of $f$ to
\begin{equation}\label{eq:defAn}
B_n=\Big\{t\in S: t+\frac{j}{N}\in S \textrm{ for exactly } n
\textrm{ integer } j{'s}\Big\}.
\end{equation}
($A_n$ is the ``folding'' of $B_n$ to $[0,\frac 1N]$ i.e.\ cutting
it to $N$ pieces, translating each one to $[0,\frac 1N]$ and taking a union).
%, by $f_n^j$ the restriction of $f$ to $A_n^j$ and by $\tilde{f}_n^j$ the function supported on $[0, \frac{2\pi }{N}]$ who's translate by ${2\pi j}/{N}$ equals to $f_n^j$.
It is enough to show that for every $n=1,\dotsc,N$ we have
\begin{equation}\label{what we need for frame}
\sum_{\l\in\L}|\langle f, e_\l\rangle|^2\geq c\|f_n\|^2-\sum_{k=1}^{n-1}\|f_k\|^2,
\end{equation}
where $c$ is a positive constant, not depending on
$f$. %\marginpar{removed the mention of C}
Indeed, the inequalities in (\ref{what we need for frame}) imply
that for any sequence of positive numbers $\{\delta_n\}_{n=1}^N$ with
$\sum\delta_n=1$ we have %\marginpar{made the $c_2$ just $c$-the 2 seemed to confuse me when reading-i hope this is ok}
\begin{align*}
\sum_{\l\in\L}|\langle f, e_{\l}\rangle|^2&=
\sum_{n=1}^N\delta_n\sum_{\l\in\L}|\langle f,
e_\l\rangle|^2\\
&\stackrel{\textrm{(\ref{what we need for frame})}}{\geq}
 \sum_{n=1}^N\delta_n\Big(c\|f_n\|^2-\sum_{k=1}^{n-1}\|f_k\|^2\Big)= \sum_{n=1}^N\Big(c\delta_n-\sum_{k=n+1}^{N}\delta_k\Big)\|f_n\|^2.
\end{align*}
Denote the last constant by $c_2$ for clarity. We get that, if the sequence $\{\delta_n\}$ satisfies%\marginpar{still confuses me the way you choose c1 here, but at least i understand why we keep choosing it differently}
\[
\delta_n>\frac 2{c_2}\sum_{k=n+1}^{N}\delta_k,\qquad\forall n\in\{1,\dotsc,N\}
\]%\marginpar{changed the requirements for $c_1$}
(essentially it needs to decrease exponentially), then for $c_1=\frac
12 c_2\min\delta_n$ we get that%\marginpar{added 2 to the requirement from $\delta_n$, then used the constant b}
\[
\sum_{\l\in\L}|\langle f, e_\l\rangle|^2\geq c_1\sum_{n=1}^N\|f_n\|^2=c_1\|f\|^2.
\]
as needed.

Hence we need to show (\ref{what we need for frame}). Fix therefore
some $n\in\{1,\dotsc,N\}$ until the end of the proof. Now, for any $x,y\in\CC$, $|x+y|^2\ge \frac12
|x|^2-|y|^2$. So,%\marginpar{changed a hence to a so}
\[
|\langle f,e_\l\rangle|^2\ge
\frac12\Big|\Big\langle\sum_{k=n}^Nf_k,e_\l\Big\rangle\Big|^2 -
\Big|\Big\langle\sum_{k=1}^{n-1}f_k,e_\l\Big\rangle\Big|^2.
\]
For brevity define $f_{\ge n}=\sum_{k=n}^Nf_k$. Summing over $\l\in\L$ gives%\marginpar{changed location of $\L\subset\Z$}
\begin{align*}
\sum_{\l\in\L}|\langle f, e_\l\rangle|^2
&\ge \frac{1}{2}\sum_{\l\in\L}|\langle f_{\ge n},e_\l\rangle|^2-
  \sum_{\l\in\L}\Big|\Big\langle \sum_{k=1}^{n-1}f_k, e_\l\Big\rangle\Big|^2\\
&\stackrel{\textrm{\clap{$(*)$}}}{\geq}
  \frac{1}{2}\sum_{\l\in\L}|\langle f_{\ge n}, e_\l\rangle|^2-
  \Big\|\sum_{k=1}^{n-1}f_k\Big\|^2\\
&\stackrel{\textrm{\clap{$(**)$}}}{=}
  \frac{1}{2}\sum_{\l\in\L}|\langle f_{\ge n}, e_\l\rangle|^2-
  \sum_{k=1}^{n-1}\|f_k\|^2
\end{align*}
where $(*)$ is because $\L\subset\Z$ and $(**)$ since $f_k$ have disjoint supports. Hence, to obtain (\ref{what we need for frame}) it remains to show that
\begin{equation}\label{more of what we need for frame}
\sum_{\l\in\L}|\langle f_{\ge n}, e_\l\rangle|^2\geq   c\|f_n\|^2
\end{equation}%\marginpar{added "Fix n" at the end of this paragraph}
where $c$ is a positive constant not depending on $f$. %Fix therefore some $n\in\{1,...,N\}$.

Recall that $\L=\cup\L_j$ and $\L_j\subseteq N\Z+j$. For any $\l\in N\Z+j$ we have
\begin{align}
\langle f_{\ge n},e_{\l}\rangle
  &=\int_0^1 f_{\ge n}(t)\overline{e_{\l}}(t)\,dt
  =\int_0^{1/N}\sum_{l=0}^{N-1}f_{\ge n}\Big(t+\frac
  lN\Big)e_{-\lambda}\Big(t+\frac lN\Big)\,dt \nonumber\\
&=\int_0^{1/N}h_j(t)e(-\l t)
  =\langle h_j,e_\lambda\rangle.\label{eq:dumb}
\end{align}
where%\marginpar{I changed the order of the first half page here, and removed some stuff - I really think that here the extensive detail made it more confusing. in the def of h i moved the $-1$ - i think this is how you prefer it}
\[
h_j(t)=\1_{A_{\ge n}}(t)\cdot \sum_{l=0}^{N-1}f_{\ge n}\Big(t+\frac lN\Big)e\Big(-\frac {jl}N\Big).
\]
Fix $j\leq n$. Since $\L_j$ is a Riesz basis for $A_{\ge j}$ and since $h_j$
is supported on $A_{\ge n}\subset A_{\ge j}$ we have%\marginpar{I changed the location of j}
\begin{equation}\label{eq:Riesz}
\sum_{\l\in\L_j}|\langle f_{\ge n},e_{\l}\rangle|^2\stackrel{\textrm{(\ref{eq:dumb})}}{=}\sum_{\l\in\L_j}|\langle
h_j,e_\l\rangle|^2
\ge c ||h_j||^2
\end{equation}
where $c$ is the Riesz constant of $\L_j$. %\marginpar{changed the
                                %line with $\L_n\subseteq \L_j$}
Summing over $j$ gives%\marginpar{changed $B_n$ to $A_n$}
\begin{align}\label{eq:Fh}
\sum_{\l\in\L}|\langle f_{\ge n},e_\l\rangle|^2&\ge
\sum_{j=1}^n\sum_{\l\in\L_j}|\langle f_{\ge n},e_\l\rangle|^2\ge\nonumber\\
%\sum_{j=1}^n\sum_{\l\in\L_n+j}|\langle f_{\ge n},e_\l\rangle|^2
&\stackrel{\textrm{(\ref{eq:Riesz})}}{\ge}
c\sum_{j=1}^n||h_j||^2\ge c\sum_{j=1}^n||h_j\cdot\1_{A_n}||^2.
\end{align}
%where the last inequality is due to the fact that $e_j$ is unimodular.

%\marginpar{I changed most of this paragraph, I guess it needs more editing}Thus we need only compare $\sum||e_{j}h_j\cdot 1\!\!1_{B_n}||^2$ and $||f_n||^2$. However,
For every particular $t\in A_n$ the values of $\{h_j(t)\}_j$ are given by
applying the $n\times N$ matrix $L=\{e(-jl/N)\}_{j,l}$ to the vector
$\{f_{\ge n}(t+l/N)\}_l$. Now, $t\in A_n$ so exactly $n$ different values of this vector are non-zero.
Considering only these values we may think of $L$ as an
$n\times n$
Vandermonde matrix which is invertible because the numbers $e(-l/N)$
are different. Let $C$ be a bound for the norm of the inverse over all such $n\times n$ sub-matrices of $L$. We get
\[
\sum_{j=1}^n |h_j(t)|^2 \ge \frac 1C\sum_{l=0}^{N-1} \Big|f_{\ge n}\Big(t+\frac
lN\Big)\Big|^2
\]
which we integrate over $t\in A_n$ to get
\[
\sum_{j=1}^n ||h_j\cdot\1_{A_n}||^2\ge c\sum_{l=0}^{N-1} \int_{A_n}\Big|f_{\ge n}\Big(t+\frac
lN\Big)\Big|^2dt=c||f_{n}||^2.
\]
%(where the last inequality is due to the fact that the $f_n$ have
%disjoint supports).
With (\ref{eq:Fh}) we get (\ref{more of what we need for frame}),
which in turn gives (\ref{what we need for frame}) and finally that
$\L$ is a frame.

\medskip
\noindent\textbf{Riesz sequence.} We now show that $E(\L)$ is a Riesz
sequence in $L^2(S)$, i.e.\ that for any sequence $a_\l\in l^2(\L)$,%\marginpar{changed L to l and c to b}
\[
\Big\Vert\sum_{\l\in\L}a_\l e_\l\Big\Vert^2\ge
c\sum_{\l\in\L}|a_\l|^2
\]
(again, the other inequality in (\ref{riesz sequence}) follows from
$S\subset[0,1]$ and $\L\subset\Z$).
% Let $\{a_{\l}\}\in l^2(\L)$.
As in the first part, it is enough to show that for every $n=1,\dotsc,N$ we have
\begin{equation}\label{what we need for r-s}
\int_{S}\Big|\sum_{\l\in\L}a_{\l} e_\l\Big|^2\geq c\sum_{\l\in\L_n}|a_{\l}|^2-\sum_{j=n+1}^{N}\sum_{\l\in\L_j}|a_{\l}|^2.
\end{equation}
To this end choose $n\in\{1,\dotsc,N\}$. We have,%\marginpar{changed some k's to j's}
\begin{align*}
\int_{S}\Big|\sum_{\l\in\L}a_{\l} e_\l\Big|^2dt&\geq
\frac{1}{2}\int_{S}\Big|\sum_{j=1}^n\sum_{\l\in\L_j}a_{\l} e_\l\Big|^2-\int_{S}\Big|\sum_{j=n+1}^N\sum_{\l\in\L_j}a_{\l} e_\l\Big|^2\\
&\ge\frac{1}{2}\int_{S}\Big|\sum_{j=1}^n\sum_{\l\in\L_j}a_{\l} e_\l\Big|^2-\sum_{j=n+1}^{N}\sum_{\l\in\L_j}|a_{\l}|^2
\end{align*}
where the second inequality is due to $S\subset[0,1]$ and $\L\subset\Z$. Denote for
brevity
\[
f=\1_S\cdot \sum_{j=1}^n \sum_{\l\in\L_j}a_\l e_\l
\]
and get that to prove (\ref{what we need for r-s}) it remains to show that
\begin{equation}\label{more we need for r-s}
\int_{S}|f(t)|^2\,dt\geq c\sum_{\l\in\L_n}|a_{\l}|^2.
\end{equation}
%Since $\int_S\ge\int_{B_{\ge n}}$ (recall the definition of $B_{\ge n}$,
%(\ref{eq:defAn})), reduce the problem further to showing
%\begin{equation}\label{more we need for r-s}
%\int_{B_{\ge n}}|f(t)|^2\,dt\geq c\sum_{\l\in\L_n+n}|a_{\l}|^2.
%\end{equation}
As in the previous case, we need to translate the problem to $A_{\ge n}$ in
order to use the assumption of the lemma. %\marginpar{I changed the
                                %notation $e(t)$ to $e_1(t)$}
Fix therefore some $t\in A_{\ge n}$ and some $l\in\{0,\dotsc,N-1\}$ such
that $t+l/N\in S$ and write
\begin{align*}
f\Big(t+\frac lN\Big) &=
\sum_{j=1}^n\sum_{\l\in\L_j} a_{\l} e_{\l}\Big(t+\frac lN\Big)=
\sum_{j=1}^n\sum_{\l\in\L_j}a_{\l} e\Big(\l t+\frac
{\l l}N\Big)=\\
&=\sum_{j=1}^n e(jl/N) \sum_{\l\in\L_j}a_{\l} e_\l(t)
\end{align*}
where the last equality is because $\L_j\subset N\Z+j$. Again we see that the vector $\{f(t+l/N)\}_l$ (where $l$ runs only
over the values for which $t+l/N\in S$, say there are $k$ such values)
is given by applying the $k\times n$ %\marginpar{I changed $n\times k$ to $k\times n$ and $\{e(jl/N)\}_{j,l}$ to $\{e(jl/N)\}_{l,j}$}
matrix $\{e(jl/N)\}_{l,j}$ to the vector
$\big\{\sum_{\l\in\L_j}a_{\l} e_\l(t)\big\}_j$. Again this
matrix has full rank because $k\ge n$ and any $n\times n$ minor is an invertible
Vandermonde matrix. Set $C$ to be a bound for the norms of the
inverses over all these minors and get, for all $t\in A_{\ge n}$,
\begin{equation}\label{eq:onet}
\sum_{l=0}^{N-1}\Big|f\Big(t+\frac lN\Big)\Big|^2
\ge \frac 1C\sum_{j=1}^n\Big|\sum_{\l\in\L_j}a_{\l} e_\l(t)\Big|^2
\ge \frac 1C\Big|\sum_{\l\in\L_n}a_{\l} e_\l(t)\Big|^2.
\end{equation}
Integrate over $t$ and get
\begin{align*}
\int_S|f(t)|^2\,dt
& \ge \int_{B_{\ge n}}|f(t)|^2\,dt
=\int_{A_{\ge n}}\sum_{l=0}^{N-1}\Big|f\Big(t+\frac lN\Big)\Big|^2dt\\
& \stackrel{\textrm{\clap{(\ref{eq:onet})}}}{\ge}
\frac 1C\int_{A_{\ge n}}\Big|\sum_{\l\in\L_n}a_{\l}
e_\l(t)\Big|^2.
\end{align*}
%% \[
%% \sum_{l=1}^n\Big|\sum_{\l\in\L_{\le n}}a_\l e_\l\Big(t+\frac
%%   lN\Big)\Big|^2
%% \]
%% Write now
%% \begin{align*}
%% \int_{S}\Big|\sum_{j=1}^n&\sum_{\l\in\L_j+j}a_{\l} e_\l\Big|^2\geq \int_{A_n}\Big|\sum_{j=1}^n\sum_{\l\in\L_j+j}a_{\l} e_\l\Big|^2dt\\
%% &=\sum_{l=0}^{N-1}\int_{A^l_n}|\sum_{j=0}^{n-1}\sum_{\l\in\L_{(j+1)}}a_{\l} e^{i(\l+j) t}|^2dt\\
%% &=\sum_{l=0}^{N-1}\int_{A^l_n-\frac{2\pi l}{N}}|\sum_{j=0}^{n-1}\sum_{\l\in \L_{(j+1)}}a_{\l} e^{i(\l+j) (t+\frac{2\pi l}{N})}|^2dt.
%% \end{align*}
%% Since $\L_j\subseteq N\Z$ for every $j=0,...,N-1$, this is equal to
%% \[
%% \sum_{l=0}^{N-1}\int_{A^l_n-\frac{2\pi l}{N}}|\sum_{j=0}^{n-1}e^{i\frac{2\pi jl}{N}}\sum_{\l\in \L_{(j+1)}}a_{\l} e^{i(\l+j)t}|^2dt.
%% \]
%% Note that $\cup_{l=0}^{N-1}(A^l_n-2\pi l/N)=E_n$ and moreover, that every $t\in E_n$ belongs to exactly $n$ sets of the form $A^l_n-2\pi l/N$.
%% We use once more the invertibility of van-der-monde matrices to find that the last expression is bigger then some constant multiplying
%% \[
%% \sum_{j=0}^{n-1}\int_{E_n}|\sum_{\l\in \L_{(j+1)}}a_{\l} e^{i(\l+j)t}|^2dt\geq
%% \]
%% \[
%% \int_{E_n}|\sum_{\l\in\L_n}a_{\l} e^{i\l t}|^2dt.
%% \]
Since $E({\L_n})$ is a Riesz basis over $A_{\ge n}$, this is $\ge
c\sum_{\l\in\L_n}|a_{\l}|^2$. This shows
(\ref{more we need for r-s}), hence (\ref{what we need for r-s}), and completes the proof.
\end{proof}

Let us remark that a similar approach of dissection and translation was
used in \cite{OU09}.

%\begin{remark}
%Let $S_1$ and $S_2$ be two sets of positive measure. Assume that the process described in Lemma 1, using the partition of $[0,2\pi]$ into $N$ intervals, yields %two sequences $E(\L_1)$ and $E(\L_2)$  which are Riesz bases over $S_1$ and $S_2$ respectively. If $S_1$ is contained in an $N$-commeasurable set of measure %$k/N$ and $S_2$ contains an $N$-commeasurable set of measure $k/N$, then $\L_1\subseteq \L_2$.
%\end{remark}
%\begin{proof}
%Indeed, in this case we have for the set $S_1$
%\[
%E_j=\emptyset \qquad \forall k<j\leq N
%\]
%So
%\[
%\L_1\subseteq \cup_{j=1}^{k}(N\Z+j-1).
%\]
%On the other hand, for $S_2$ we have
%\[
%[0,\frac{2\pi}{N}]\subseteq E_j\qquad \forall 1\leq j\leq K
%\]
%So
%\[
%\cup_{j=1}^{k}(N\Z+j-1)\subseteq \L_2.
%\]
%\end{proof}

\section{Single intervals}
In this section we use lemma \ref{lem:main} to obtain some corollaries regarding Riesz bases of exponentials over single intervals. Our construction for Riesz bases over finite unions of intervals will be done in much the same way. We end this section with Claim \ref{generic case}, which is a simplified version of our main result: it states that our main result holds in the ``generic'' case.

The following lemma was proved in \cite{S95} using Avdonin's
theorem. As promised, we give an elementary proof.
\begin{lemma}\label{one interval}
Let $S\subseteq [0,1]$ be an interval.
Then there exists a sequence $\L\subseteq\Z$ such that $E(\L)$ is a Riesz basis for $L^2(S)$.
\end{lemma}

Before starting with the proof of lemma \ref{one interval}, let us do
a special case.%We start with the following observation.
\begin{claim}\label{using stability claim} %\marginpar{added "stability"}
Let $\eta$ be a constant satisfying the conditions of the Paley-Wiener stability
theorem for $S=[0,1]$ and the Riesz basis $\Z$. There exists a sequence $\L\subseteq \Z$ such that $E(\L)$ is a Riesz basis for $L^2[0,\eta]$.
\end{claim}
\begin{proof}
Applying an appropriate transformation from $L^2[0,1]$ to
$L^2[0,\eta]$, we find that the system $E((1/\eta)\Z)$ is a Riesz
basis (in fact, an orthogonal basis) over $[0,\eta]$. Moreover, the Paley-Wiener stability theorem implies that if $\L=\{\l_n\}$ satisfies $|n/\eta -\l_n|\leq 1$ then $E(\L)$ is also a Riesz basis over $[0,\eta]$. This gives the claim.
\end{proof}

\begin{proof}[Proof of lemma \ref{one interval}]
We may assume that $S=[0,b]$.  Fix $\eta$ as in claim (\ref{using
  stability claim}). First, we note that there exists a number $N$ such that
$\{Nb\}\le \eta$ (recall that $\{\cdot\}$ denotes the fractional value). %0\leq b-\frac{k}{N}\leq\frac{\eta}{N}.
Indeed, this is clear if $b$ is rational and follows from the density of $\{Nb\}_{N\in\N}$ if $b$ is not rational.

Recall the definition of the sets $A_{\ge n}$, (\ref{eq:defAgen}). %Consider the partition of $[0,1]$ into the intervals $[0,j/N]$, $0\leq j<N$.
For our $S$ and $N$, the sets $A_{\geq n}$ are simply %\marginpar{changed a bN to an Nb}
\begin{align*}
A_{\geq n} &= \Big[0,\frac{1}{N}\Big]&&  1\leq n\leq \lfloor Nb\rfloor\\
A_{\geq {\lfloor Nb\rfloor+1}}&=\bigg[0,\frac{\{Nb\}}{N}\bigg]\\
A_{\geq n}&=\emptyset&&  \lfloor Nb\rfloor +2\leq n\leq N.
\end{align*}
To apply lemma \ref{lem:main} we need to demonstrate Riesz bases in
$N\Z$ for these sets. For $[0,1/N]$ we take $N\Z$ itself, and for
$[0,\{bN\}/N]$ we can apply claim \ref{using stability claim} (after
scaling by $N$) since $\{bN\}<\eta$. Hence lemma \ref{lem:main}
applies and we are done.
\end{proof}

In particular, lemma \ref{one interval} implies that for ``most'' of the sets $S$, which are finite unions of intervals, there exists a Riesz basis of exponentials. We add a short proof for this claim, as it is simpler than the proof of our general result:

\begin{claim}\label{generic case}
Let $S\subset [0,1]$ be a finite union of intervals $S=\cup_{j=1}^L[a_j, b_j]$. If the numbers $\{a_1,\dotsc,a_L,b_1,\dotsc,b_L,1\}$ are linearly independent over the rationals then there exists a sequence $\L\subset\Z$ such that $E(\L)$ is a Riesz basis over $S$.
\end{claim}

\begin{proof}
Due to rational independence, the vectors
$(\{Na_1\},\dotsc,\{Na_L\},\{Nb_1\},\dotsc,\linebreak[4]\{Nb_L\})$ are dense in
$[0,1]^{2L}$ and in particular there exists an $N$ such that % positive integers $ k_1,\dotsc,k_L,\linebreak[1]l_1,\dotsc,l_L$ and $N$, such that
% \[\frac{k_j}{N}+\frac{1}{2N}\leq a_j<\frac{k_j+1}{N}\qquad\textrm{and}\qquad \frac{k_j}{N}\leq b_j<\frac{k_j}{N}+\frac{1}{2N}.
% \]
\[
\{Na_1\},\dotsc,\{Na_L\}<\frac 12<\{Nb_1\},\dotsc,\{Nb_L\}.
\]
%Consider the partition of $[0,1]$ into the intervals $[0,2\pi j/N]$.
Now, the function $\Phi(t)=|\{j:t+j/N\in S\}|$ increases at every
$\{Na_i\}/N$ and decreases at every $\{Nb_i\}/N$ so the restrictions above
mean that it is increasing on $[0,1/2N]$ and decreasing on
$[1/2N,1/N]$. In particular $A_{\ge n}=\{t:\Phi(t)\ge n\}$ is an interval
for every $n$.
%The sets $A_{\geq n}$ which correspond to this $N$ are all single intervals (when considered as subsets of $1/N\cdot\mathrm{T}$).
We use lemma \ref{one interval} to find a Riesz basis in $\Z$ for each $NA_{\ge
  n}$, rescale to get a Riesz basis in $N\Z$ for $A_{\ge n}$ and
finally apply lemma \ref{lem:main} to construct a Riesz basis for $S$.
\end{proof}

\section{Finite union of intervals}

\subsection{Auxiliary results} The proof of the theorem in the general
case requires some understanding of the possible orders $\{Na_i\}$ may
hold (compare to the proof of claim \ref{generic case}).
%relays on some number theoretic arguments.
First, we mention a simple observation.

\begin{claim}\label{no L intervals}
Let $L$ be a positive integer and $a_1\leq\dotsb\leq a_{L}$, $b_1\leq\dotsb\leq b_{L}\in [0,1]$ be $2L$ numbers. Set
\[
\Phi(t)=\sum_{l=1}^L\1_{[0,b_l]}(t)+\sum_{l=1}^L\1_{[a_l,1]}(t)\qquad A_{\geq n}=\Phi^{-1}[n,2L];\quad 1\leq n\leq 2L.
\]
Then the sets $A_{\geq n}$ are all unions of at most $L$ intervals
(when considered cyclically). Moreover, if there exists some $n$ such that $A_{\geq
  n}$ is a union of exactly $L$ intervals then the sequences $a_l$ and
$b_l$ interlace, i.e.\ either $a_1\le b_1\le a_2\le \dotsb$ or $b_1\le
a_1\le b_2 \le \dotsb$
\end{claim}
\begin{proof}%\marginpar{ I removed the if and only if condition, we don't use it. added to the claim "Moreover" and "the" in the proof changed "sequence does not" to "sequences do not"}
%Set $a_0=0$ and $b_0=1$.
The function $\Phi$ increases only in the $a_l$ and decreases only in
the $b_l$. Hence each set $A_{\geq n}$ is a union of cyclic intervals of the
form $[a_l,b_j]$, $1\leq l,j\leq L$. Since there are only $2L$
potential end-points, the sets $A_{\geq n}$ are all unions of at most $L$ intervals.

Assume that the sequences do not interlace. This means that, for some $1\leq l<L$, no $b_j$ lies in the interval $[a_l, a_{l+1}]$. It follows that a set $A_{\geq n}$ which contains an interval with left point at $a_l$, does not contain an interval with left point at $a_{l+1}$. Hence, the sets $A_{\geq n}$ are all unions of at most $L-1$ intervals.
\end{proof}

%We will use also the following Lemma.
With claim \ref{no L intervals}, it is natural to investigate
interlacement. Let us state our main claim.

\begin{lemma}\label{ihatethislemma}
Let $a_{1},\dotsc,a_{L},b_{1},\dotsc,b_{L}\in\mathbb{R}$ be all different.
Then there exist infinitely many $N\in\mathbb{N}$ such that $\{Na_{i}\}$
and $\{Nb_{i}\}$ do not interlace.
%Let $A=\{a_1,\dotsc,a_L\}\subset \R$.
%Then there exist infinitely many $N\in\mathbb{N}$ such that no
%element %\marginpar{removed "there are"}
%of the form $\{Na\}$, $a\in A$ lies strictly between $\{Na_{1}\}$ and $\{Na_{2}\}$.
\end{lemma}
Again, the term ``interlace'' should be understood cyclically: it
means that for some permutations $\sigma$ and $\tau$ of $\{1,\dotsc,L\}$
and for some $x\in[0,1]$,
\[
\{Na_{\sigma(1)}+x\}\le\{Nb_{\tau(1)}+x)\}\le\{Na_{\sigma(2)}+x\}\le\dotsb\le\{Nb_{\tau(L)}+x\}
\]
($x$ can always be taken to be $-Na_{\sigma(1)}$).
%Again, the term ``between'' is to be understood cyclically. Since we do
%not make a claim about the order of $\{Na_{1}\}$ and $\{Na_{2}\}$
%then the lemma should be understood to mean that when you remove from
%the circle $\T$ the points $\{Na_{1}\}$ and $\{Na_{2}\}$,
%one of the remaining intervals is empty of points of the form
%$\{Na\}$, $a\in A$.

Since the proof of Lemma \ref{ihatethislemma} has a different flavour
than the rest of the argument, we postpone it to the next section.

\subsection{Proof of the main result}

We are now ready to prove the theorem. %\marginpar{I left the statement of the theorem, but I think it should just be moved to the intro}
%\begin{theorem}
%Let $S\subseteq [0,1]$ be a finite union of intervals intervals.
%Then there exists a sequence $\L\subseteq\Z$ such that $E(\L)$ is a Riesz basis over $S$.
%\end{theorem}
\begin{proof} We prove the ``moreover clause'' of the theorem
  i.e.\ assume that $S\subset[0,1]$ and find a $\L\subset\Z$ such that
  $E(\L)$ is a Riesz basis for $S$. We use induction on $L$, the number of intervals $S$ is constructed of. The case $L=1$ was checked in lemma \ref{one interval}. Assume that the theorem holds for all positive integers smaller then $L$.
 Denote $S=\cup_{l=1}^L[a_l,b_l]$. By lemma \ref{ihatethislemma} there
 exist arbitrarily large $N$ such that
%there are no elements
%of the form $\{Nb_l\}$ strictly between $\{Na_{1}\}$ and $\{Na_{2}\}$,
%when considered cyclically. This implies that
the sequences
$\{Na_l\}$ and $\{Nb_l\}$ do not interlace. Require also that $N$ is so big that every interval $[a_l,b_l]$ contains at least one element of the form $k/N$.

%Consider the partition of $[0,1]$ into the intervals $[0, j/N]$.
%\marginpar{ I must say that I think this paragraph is another case of 'less explanation is more clear' but i hesitate to change. Your choice}
Examine now the sets $A_{\ge n}$. Because each interval $[a_l,b_l]$
contains an element of $\frac 1N \Z$, then $[a_l,b_l]$ is partitioned
into translates of
$[\{Na_l\}/N,1/N]$, $[0,\{Nb_l\}/N]$ and $[0,1/N]$. We now apply
claim \ref{no L intervals} with
\[
a_i^\textrm{claim \ref{no L intervals}}=\{Na_i\}\qquad
b_i^\textrm{claim \ref{no L intervals}}=\{Nb_i\}\qquad
\]
and get that the $A_{\ge n}^\textrm{claim \ref{no L intervals}}$ have
at most $L-1$ intervals. But each $A_{\ge n}$ is $\frac 1N\times$ some
$A_{\ge n}^\textrm{claim \ref{no L intervals}}$ (not necessarily the
same $n$, because of the pieces $[0,1/N]$, but this is not
important).
%By claim \ref{no L intervals}, the sets $E_{\geq n}$ which correspond
%to this partition are unions of at most $(L-1)$ intervals (when
%considered as subsets of $1/N\cdot\mathrm{T}$).
Hence, we may apply our inductive assumption to construct Riesz bases
in $\Z$
for all the sets $NA_{\ge n}$, rescale to get Riesz bases in $N\Z$
for %\marginpar{changed lemma to Lemma}
the sets $A_{\ge n}$ and then apply Lemma \ref{lem:main} to finish the induction.
\end{proof}

 In a similar way one can show the following corollaries of the theorem, we omit the proofs.
\begin{corollary}%\marginpar{added L intervals - else it is not clear who is L}
Let $S\subseteq [0,1]$ be a union of $L$ intervals and fix $N>0$. Assume that $m/N\leq|S|<(m+1)/N$ where $m$ is a positive integer. Then $\L$ of theorem 1 can be chosen to satisfy
    \[
    \cup_{j=0}^{m-2L-1}(N\Z+j)\subseteq \L \subseteq \cup_{j=0}^{m+2L}(N\Z+j).
    \]
\end{corollary}
\begin{corollary}
 Let $S_1\subset S_2\subset\dotsb\subset S_K\subset [0,1]$ be a family of sets, all of which are finite unions of intervals. Then there exist sequences $\L_1\subset \L_2\subset\dotsb\subset \L_K\subset \Z$ such that $E(\L_k)$ is a Riesz basis over $S_k$.
\end{corollary}
\begin{remark}It is not difficult to check that the Riesz bases we
  construct all have bounded discrepancy, namely, if the measure of
  $S$ is $\alpha$, then the Riesz basis $\L$ we construct has the
  property that for any interval $I\subset\R$,
\[
\big| |\L\cap I|-\alpha |I|\big| \le C.
\]
Of course, this is due to us eventually relying on Paley-Wiener
stability. Riesz bases constructed by applying Paley-Wiener to scaled
copies of $\Z$ have bounded discrepancy, and this property is
preserved throughout.
\end{remark}

\section {Proof of lemma \ref{ihatethislemma}}

%In this section we prove lemma \ref{ihatethislemma}. With this the
%proof of the theorem will be complete. Recall that it
%claims that it is not possible for $\{Na_j\}$ and $\{Nb_j\}$ to
%interlace for almost all $N$.
In this section we prove lemma \ref{ihatethislemma}, with this the %\marginpar{I changed the first paragraph}
proof of the theorem will be complete. Our original proof of this lemma was long and rather
cumbersome (the interested reader can see it in the first arXiv version of
the paper), the proof we present here was shown to us by Fedor Nazarov.

Recall that Lemma \ref{ihatethislemma}
claims that it is not possible for $\{Na_j\}$ and $\{Nb_j\}$ to
interlace for almost all $N$.
The proof revolves around the following quantities:
\[
s_{N}=\bigg|\sum_{j=1}^{L}e(Na_{j})-\sum_{j=1}^{L}e(Nb_{j})\bigg|^{2}\qquad S_{K}=\sum_{N=1}^{K}s_{N}.
\]
We will show that $S_{K}$ is large, while whenever $\{Na_{i}\}$
and $\{Nb_{i}\}$ interlace, $s_{N}$ is small, which will provide
a contradiction.

We first show that $S_{K}$ is large. We open the square in the definition
of $s_{N}$, using the notations
\[
\xi_{j}=\begin{cases}
e(a_{j}) & j\in\{1,\dotsc,L\}\\
e(b_{j-L}) & j\in\{L+1,\dotsc,2L\}
\end{cases}\quad
\eps_{j}=\begin{cases}
1 & j\in\{1,\dotsc,L\}\\
-1 & j\in\{L+1,\dotsc,2L\}
\end{cases}
\]
and we get
\begin{align}
S_{K}
&=\sum_{j=1}^{2L}\sum_{k=1}^{2L}\eps_j\eps_k\sum_{N=1}^{K}\xi_{j}^{N}\overline{\xi_{k}}^{N}\nonumber\\
&=2LK+\sum_{j\ne
  k}\eps_j\eps_k\frac{\xi_{j}\overline{\xi_{k}}-(\xi_{j}\overline{\xi_{k}})^{K+1}}{1-\xi_{j}\overline{\xi_{k}}}
=2LK+O(1)\label{eq:SN large}
\end{align}
where the constant implicit in the $O(1)$ might depend on $L$ and
on the $\xi_{j}$ but not on $K$.

We now bound $s_{N}$ in the case that $\{Na_{i}\}$ and $\{Nb_{i}\}$
interlace. Define $\alpha=\sum e(Na_{j})-\sum e(Nb_{j})$ so that
$s_{N}=|\alpha|^{2}$ (we omit the dependency of $\alpha$ on $N$
in the notation). We assume $\alpha\ne0$ since otherwise there is
nothing to prove. Let $P$ be the map from $\mathbb{C}$ to $\mathbb{R}$
defined by $Px=\langle x,\alpha/|\alpha|\rangle_{\R^2}$ (where $x$ and $\alpha$ are thought of as points in $\R^2$),
i.e.\ a projection on
the line $\alpha \mathbb{R}$ and then a rotation to $\mathbb{R}$.
%\ $Px=\langle x,\alpha/|\alpha|\rangle$.
We get
\[
s_{N}=|P(\alpha)|^{2}=\bigg|\sum_{j=1}^{L}P(e(Na_{j}))-\sum_{j=1}^{L}P(e(Nb_{j}))\bigg|^{2}.
\]
Let now $\left\{ d_{j}\right\} _{j=1}^{2L}$ be the collection of
points $\{e(Na_j)\}\cup\{e(Nb_j)\}$ %\marginpar{I changed the "numbers" to "points" (not sure about this) and $\alpha$ to $\alpha/\|\alpha\|$} $\{e(Na_{j})\}_{j=1}^{L}\cup\{e(Nb_{j})\}_{j=1}^{L}$ 
arranged
counterclockwise starting from $\alpha/|\alpha|$.
%Let now $\left\{ d_{j}\right\} _{j=1}^{2L}$ be the collection of
%numbers $\{e(Na_{j})\}_{j=1}^{L}\cup\{e(Nb_{j})\}_{j=1}^{L}$ arranged
%counterclockwise starting from $\alpha$.
For example, if $\alpha\in\mathbb{R}^{+}$
then $0\le\arg d_{1}\le\arg d_{2}\le\dotsb$. Under the assumption
of interlacement, we see that either $d_{2j}\in\{e(Na_{k})\}$ and
$d_{2j+1}\in\{e(Nb_{k})\}$ or vice-versa and hence in both cases
\[
s_{N}=\bigg|\sum_{j=1}^{L}P(d_{2j-1})-P(d_{2j})\bigg|^{2}.
\]
The crucial observation is now that $P(d_{2j-1})-P(d_{2j})$ are positive
up to some $j_{0}$, and negative from that point on. This means that
\[
\bigg|\sum_{j=1}^{L}P(d_{2j-1})-P(d_{2j})\bigg|
\le\max\bigg\{\sum_{j=1}^{j_{0}}P(d_{2j-1})-P(d_{2j}),\sum_{j=j_{0}+1}^{L}P(d_{2j})-P(d_{2j-1})\bigg\}.%\label{eq:posneg}
\]%\end{multline}
Further, the intervals $[P(d_{2j-1}),P(d_{2j})]$ are disjoint in
both regimes (i.e.\ up to $j_{0}$ and from $j_{0}+1$ on), so both
sums on the right-hand side of the equation above are bounded by $2$.
We get
\begin{equation}
s_{N}\le4\textrm{ whenever $\{Na_{i}\}$ and $\{Nb_{i}\}$ interlace.}\label{eq:sn4}
\end{equation}
This immediately finishes the case of $L\ge3$. Indeed, if for all
$N$ but a finite number $\{Na_{i}\}$ and $\{Nb_{i}\}$ interlace
then as $K\to\infty$
\[
2LK+O(1)\stackrel{\textrm{(\ref{eq:SN large})}}{=\vphantom{\le}}S_{K}=\sum_{N=1}^{K}s_{N}\stackrel{\textrm{(\ref{eq:sn4})}}{\le}4K+O(1)
\]
(the $O(1)$ on the right-hand side is due to the finite set of bad
$N$) leading to a contradiction.

To do the case of $L=2$ with the same approach we need to strengthen
(\ref{eq:sn4}) slightly. Fix some small enough $\epsilon>0$%, such that $R=1/\epsilon^4$ is an integer
. By the Dirichlet approximation theorem,
given any integer $N$ one can find an integer $N\leq n\leq N+\lceil
\eps^{-4}\rceil$ such that $na_{1}$, $na_{2}$, $nb_{1}$, $nb_{2}$ are
all within distance smaller or equal to $\epsilon$ from the integers.
%So in this case (with $c$ being an absolute constant),
In this case we estimate $s_n$ directly from the definition
(without using $P$ as before) and get
\begin{multline*}
s_{n}^{1/2}=|e(na_{1})+e(na_{2})-e(nb_{1})-e(nb_{2})|\\
\leq |e(na_{1})-1|+|e(na_{2})-1|+|1-e(nb_{1})|+|1-e(nb_{2})|\leq C\epsilon.
\end{multline*}
Choose $\epsilon$ sufficiently small so that the right-hand side is
$<1$. We conclude that, if for all $N\le K$ we have interlacement then there
is a proportion $c_1K$ where $s_N\le 1$, and the rest are
still $\le 4$ by (\ref{eq:sn4}). So we get (under interlacement for
almost every $N$),
%if for all
%$N$ but a finite number $\{Na_{i}\}$ and $\{Nb_{i}\}$ interlace
%then as $K\to\infty$
\[
S_{K}=\sum_{N=1}^{K}s_{N}\le (1-c_1)K\cdot 4 +c_1K\cdot 1
+O(1)\le (4-c)K+O(1).
\]
leading to a contradiction as before.
%Assume $s_{N}>4-\epsilon$ for some $\epsilon>0$
%sufficiently small. Then $|P(d_{1})-P(d_{2})+P(d_{3})-P(d_{4})|>2-\epsilon$
%and a little checking of cases shows that we must have $|P(d_{i})-P(d_{i+1})|\le\epsilon$
%for one of $i\in\{1,2,3,4\}$. This does not necessarily mean that
%$d_{i}$ and $d_{i+1}$ are themselves close, since it is possible
%for the projections to be close without the numbers themselves being
%close, but some more checking of cases shows that this is imposible under $s_{N}>4-\epsilon$
%and hence we may conclude that $|d_{i}-d_{i+1}|\le C_{1}\sqrt{\epsilon}$
%for $C_{1}$ some constant independent of $\epsilon$.

%Next note that for any two real numbers $a$ and $b$, the set of
%$N$ such that $|e(Na)-e(Nb)|\le\delta$ is either of density $\delta$
%(if $a-b\not\in\mathbb{Q}$) or a linear progression with no constant
%term otherwise. We get that for some $\delta>0$ and some $c_{1}>0$
%(which may depend on the $d_{i}$),
%\begin{equation}
%|\{N\le K:\exists i,|d_{i}-d_{i+1}|\le\delta\}|\le(1-c_{1})K\label{eq:ooof}
%\end{equation}
%for $K$ sufficiently large. This means that, with $\epsilon=(\delta/C_{1})^{2}$,
%\begin{multline*}
%\smash{\sum_{N=1}^{K}}\vphantom{\sum^{K}}s_{N}\le(4-\epsilon)K+\epsilon|\{N\le K:\exists i,|d_{i}-d_{i+1}|\le C_{1}\sqrt{\epsilon}\}|\\
%\stackrel{\smash{\textrm{(\ref{eq:ooof})}}}{\le}(4-\epsilon)K+\epsilon(1-c_{1})K=(4-\epsilon c_{1})K
%\end{multline*}
%and we reach a contradiction to (\ref{eq:SN large}) similarly.\qed

\subsection{Acknowledgements}
We thank Fedor Nazarov for showing us the version of the proof of
\ref{ihatethislemma} we presented above. This research was supported by the Israel Science
Foundation, the Jesselson Foundation, and partially by the CNRS.


\begin{thebibliography}{HNP80}
\bibitem{A79}Sergei A. Avdonin, \textcyr{\char202} \textcyr{\char226\char238\char239\char240\char238\char241\char243}
\textcyr{\char238} \textcyr{\char225\char224\char231\char232\char241\char224\char245}
\textcyr{\char208\char232\char241\char241\char224} \textcyr{\char232\char231}
\textcyr{\char239\char238\char234\char224\char231\char224\char242\char229\char235\char252\char237\char251\char245}
\textcyr{\char244\char243\char237\char234\char246\char232\char233}
\textcyr{\char226} \inputencoding{koi8-r}\foreignlanguage{russian}{$L^{2}$}\inputencoding{latin9}
{[}Russian: On the question of Riesz bases of exponential functions
in $L^{2}${]}, Vestnik Leningrad Univ. 13 (1974), 5--12. English
translation in: Vestnik Leningrad Univ. Math. 7 (1979), 203--211.

\bibitem{AM99}Sergei A. Avdonin and William Moran, {\em Sampling and
  interpolation of functions with multi-band spectra and
  controllability problems}. Optimal control of partial differential
  equations (Chemnitz, 1998), 43--51, Internat. Ser. Numer. Math., 133,
  Birkh\"auser, Basel, 1999.

\bibitem{ABM07}Sergei A. Avdonin, Anna Bulanova and William Moran,
  {\em Construction of Sampling and Interpolating Sequences for
    Multi-Band Signals. The Two-Band Case.} Internat. J. Applied
  Math. Comput. Sci. 17:2 (2007), 143--156. Available at:
  \href{http://www.amcs.uz.zgora.pl/?action=paper&paper=329}{\nolinkurl{amcs.uz.zgora.pl}}

\bibitem{BK93}L. Bezuglaya and Victor E. Katsnel\textquotesingle son,
\emph{The sampling theorem for functions with limited multi-band spectrum}.
Z. Anal. Anwendungen 12:3 (1993), 511--534.

\bibitem{F74}Bent Fuglede, \emph{Commuting self-adjoint partial
differential operators and a group theoretic problem}, J. Funct. Anal.
16:1 (1974), 101--121. Available at: \href{http://www.sciencedirect.com/science/article/pii/002212367490072X}{\nolinkurl{sciencedirect.com}}

\bibitem{F01}\bysame, \emph{Orthogonal exponentials on
the ball}. Expo. Math. 19:3 (2001), 267--272. Available at: \href{http://www.sciencedirect.com/science/article/pii/S0723086901800050}{\nolinkurl{sciencedirect.com}}

\bibitem{HNP80}Sergey V. Hru\u{s}\u{c}ev, Nikolai K. Nikol\textquotesingle skii
and Boris S. Pavlov, \emph{Unconditional bases of exponentials and
of reproducing kernels}, Complex analysis and spectral theory (Lenin\-grad,
1979/1980), 214--335, Lecture Notes in Math. 864, Springer, Berlin,
1981. Available at: \href{http://www.springerlink.com/content/apx1860h14326462/}{\nolinkurl{springerlink.com}}

\bibitem{IKT03}Alex Iosevich, Nets Katz and Terence Tao, \emph{The
Fuglede spectral conjecture holds for convex planar domains}. Math.
Res. Lett. 10:5 (2003), 559--569. Available at:
\href{http://www.intlpress.com/_newsite/site/pub/pages/journals/items/mrl/content/vols/0010/0005/00019975/index.php}{\nolinkurl{intlpress}}

\bibitem{IK}Alex Iosevich and Mihalis N. Kolountzakis, {\em
  Periodicity of the spectrum in dimension one}. Analysis and PDE 6:4
  (2013), 819--827. Available at:
  \href{http://msp.org/apde/2013/6-4/p03.xhtml}{\nolinkurl{msp.org}},
  \href{http://arxiv.org/abs/1108.5689}{\nolinkurl{arXiv:1108.5689}}

\bibitem{K64}Mikahil \u I. Kadec\textquotesingle, \textcyr{\char210\char238\char247\char237\char238\char229}
\textcyr{\char231\char237\char224\char247\char229\char237\char232\char229}
\textcyr{\char239\char238\char241\char242\char238\char255\char237\char237\char238\char233}
\textcyr{\char207\char224\char235\char229\char255}-\textcyr{\char194\char232\char237\char229\char240\char224}
{[}Russian: The exact value of the Paley-Wiener constant{]}. Dokl.
Akad. Nauk SSSR 155 (1964), 1253--1254. English translation in: Soviet
Math. Doklady 5:2 (1964), 559-561.

\bibitem{K71}Victor E. Katsnel\textquotesingle son, \textcyr{\char206}
\textcyr{\char225\char224\char231\char232\char241\char224\char245}
\textcyr{\char232\char231} \textcyr{\char239\char238\char234\char224\char231\char224\char242\char229\char235\char252\char237\char251\char245}
\textcyr{\char244\char243\char237\char234\char246\char232\char233}
\textcyr{\char226} $L^{2}$ {[}Russian: Exponential bases in $L^{2}${]}.
Funkc. Anal. i Prilo\v zen. 5:1 (1971), 37--47. \href{http://www.mathnet.ru/php/archive.phtml?wshow=paper&jrnid=faa&paperid=2560&option_lang=rus}{\nolinkurl{mathnet.ru}}.
English translation in: Funct. Anal. and Appl. 5:1
(1971), 31--38. \href{http://www.springerlink.com/content/n481414380634700/}{\nolinkurl{springerlink.com}}

\bibitem{K96} \bysame, {\em Sampling and interpolation for functions with
multi-band spectrum: the mean-periodic continuation
method}. Wiener-Symposium (Grossbothen, 1994), 91--132,
Synerg. Syntropie Nichtlineare Syst., 4, Verlag Wiss. Leipzig,
Leipzig, 1996.

\bibitem{K53}Arthur Kohlenberg, \emph{Exact Interpolation of
Band-Limited Functions}. J. Appl. Phys. 24:12 (1953), 1432--1436.
Available at: \href{http://link.aip.org/link/doi/10.1063/1.1721195}{\nolinkurl{aip.org}}

\bibitem{KM06}Mihail N. Kolountzakis and M\'at\'e Matolcsi, \emph{Tiles
with no spectra}. Forum Math. 18:3 (2006), 519--528. Available at:
\href{http://www.degruyter.com/view/j/form.2006.18.issue-3/forum.2006.026/forum.2006.026.xml?format=INT}{\nolinkurl{degruyter.com}},
\href{http://arxiv.org/abs/math/0406127}{\nolinkurl{arXiv:0406127}}

\bibitem{KL11}Gady Kozma and Nir Lev, \emph{Exponential Riesz
bases, discrepancy of irrational rotations and BMO}. J. Fourier Anal.
and Appl. \textbf{17:5} (2011), 879--898. \href{http://www.springerlink.com/content/35g5047463711323/}{\nolinkurl{springerlink.com}},
\href{http://arxiv.org/abs/1009.2188}{\nolinkurl{arXiv:1009.2188}}

\bibitem{L67a} Henry J. Landau, \emph{Necessary density conditions
  for sampling and interpolation of certain entire functions}. Acta
  Math. \textbf{117} (1967) 37--52. Available at: \href{http://www.springerlink.com/content/22h1h1514x501740/}{\nolinkurl{springerlink.com}}

%\bibitem{L67b} \bysame, \emph{Sampling, data transmission,
%  and the Nyquist rate}. Proc. IEEE \textbf{55:10} (1967),
%  1701--1706. Available at: \href{http://dx.doi.org/10.1109/PROC.1967.5962}{\nolinkurl{ieee.org}}

\bibitem{L12}Nir Lev, \emph{Riesz bases of exponentials on multiband
spectra}. Proc. Amer. Math. Soc. 140:9 (2012), 3127--3132. Available
at: \href{http://www.ams.org/journals/proc/2012-140-09/S0002-9939-2012-11138-4/home.html}{\nolinkurl{ams.org}},
\href{http://arxiv.org/abs/1101.3894}{\nolinkurl{arXiv:1101.3894}}

\bibitem{L40} Norman Levinson, \emph{Gap and Density Theorems}. Amer. Math.
Soc. Colloquium Publications, vol. 26. Amer. Math. Soc., New York, 1940.

\bibitem{LS96}Yurii Lyubarskii and Ilya M. Spitkovsky, {\em Sampling and
  interpolation for a lacunary spectrum}. Proc. Roy. Soc. Edinburgh
  Sect. A 126:1 (1996), 77--87. \href{http://dx.doi.org/10.1017/S0308210500030602}{\nolinkurl{cambridge.org}}

\bibitem{LS97}Yurii Lyubarskii and Kristian Seip, \emph{Sampling
and interpolating sequences for multi\-band-limited functions and exponential
bases on disconnected sets}. J. Fourier Anal. Appl. 3:5 (1997), 597--615.
Available at: \href{http://www.springerlink.com/content/d44064614113383u/}{\nolinkurl{springerlink.com}}

\bibitem{MM10}Basarab Matei and Yves Meyer, \emph{Simple quasicrystals
are sets of stable sampling. }Complex Var. Elliptic Equ\emph{. }55:8--10
(2010), 947--964. Available at: \href{http://www.tandfonline.com/doi/abs/10.1080/17476930903394689}{\nolinkurl{tandfonline.com}}

\bibitem{NO}Shahaf Nitzan and Alexander Olevskii, \emph{Revisiting
  Landau's density theorems for Paley-Wiener
  spaces}. C. R. Acad. Sci. Paris 350:9--10 (2012),
  509--512. \href{http://www.sciencedirect.com/science/article/pii/S1631073X12001392}{\nolinkurl{sciencedirect.com}}

\bibitem{OU09}Alexander Olevskii and Alexander Ulanovskii,
\emph{Uniqueness sets for unbounded spectra}. C. R. Acad. Sci. Paris
349:11--12 (2011), 679--681. Available at:
\href{http://www.sciencedirect.com/science/article/pii/S1631073X11001476}{\nolinkurl{sciencedirect.com}}

%  \emph{Interpolation in Bernstein and Paley-Wiener spaces}. J. Funct. Anal.
%  256:10 (2009), 3257--3278. \href{http://www.sciencedirect.com/science/article/pii/S0022123608004023}{\nolinkurl{sciencedirect.com}}

\bibitem{PW34}Ramond E. A. C. Paley and Norbert Wiener,
  \emph{Fourier transforms in the complex domain}.
  Amer. Math. Soc. Colloquium Publications vol. 19. Amer. Math. Soc., New York, 1934.

\bibitem{P79}Boris S. Pavlov, \textcyr{\char193\char224\char231\char232\char241\char237\char238\char241\char242\char252}
\textcyr{\char241\char232\char241\char242\char229\char236\char251}
\textcyr{\char253\char234\char241\char239\char238\char237\char229\char237\char242}
\textcyr{\char232} \textcyr{\char243\char241\char235\char238\char226\char232\char229}
\textcyr{\char204\char224\char234\char229\char237\char245\char224\char243\char239\char242\char224}
{[}Russian: The basis property of a system of exponentials and the
condition of Muckenhoupt{]}. Dokl. Akad. Nauk SSSR 247 (1979), 37--40.
English translation in: Soviet Math. Dokl. 20 (1979), 655--659.

\bibitem{S95}Kristian Seip, \emph{A simple construction of exponential
bases in $L^{2}$ of the union of several intervals}. Proc. Edinburgh
Math. Soc. 38:1 (1995), 171--177. Available: \href{http://dx.doi.org/10.1017/S0013091500006295}{\nolinkurl{cambridge.org}}

\bibitem{S00}Robert S. Strichartz, \emph{Mock Fourier series
and transforms associated with certain Cantor measures}. J. Anal.
Math. 81 (2000), 209\textendash{}238. Available at: \href{http://www.springerlink.com/content/pp561301214h77h9/}{\nolinkurl{springerlink}}

\bibitem{T04}Terence Tao, \emph{Fuglede's conjecture is false
in 5 and higher dimensions}. Math. Res. Lett. 11:2 (2004), 251\textendash{}258.
Available at: \href{http://www.intlpress.com/_newsite/site/pub/pages/journals/items/mrl/content/vols/0011/0002/00020021/index.php}{\nolinkurl{intlpress.com}}

\bibitem{Y01} Robert M. Young, \emph{An introduction to
  nonharmonic Fourier series}. Revised first edition. Academic Press,
  Inc., San Diego, CA, 2001.

\end{thebibliography}
\end{document}